\documentclass[11pt,twoside]{article}
\usepackage[a4paper,margin=2cm]{geometry}
\usepackage{cite,url,amsthm,amsfonts,amssymb,hyperref}
\usepackage{times}
\usepackage{amsfonts}
\usepackage{amssymb}  
\usepackage{amsmath}
\usepackage{graphicx} 
\usepackage[utf8]{inputenc}
\usepackage[T1]{fontenc}     

\theoremstyle{plain}  
\newtheorem{theorem}{Theorem}
\newtheorem{lemma}[theorem]{Lemma}
\newtheorem{proposition}[theorem]{Proposition}
\newtheorem{corollary}[theorem]{Corollary}

\date{}

\title{Irreducible Triangulations of Surfaces with Boundary%
  \thanks{This work was done during the first author's internship at
    \'Ecole normale sup\'erieure.  The internship was funded by the
    \emph{Agence Nationale de la Recherche} under the \emph{Triangles}
    project of the \emph{Programme blanc} ANR-07-BLAN-0319.}}

\author{Alexandre Boulch\thanks{\'Ecole Polytechnique (member of
    ParisTech), Palaiseau, France,
    \protect\url{alexandre.boulch@polytechnique.edu}}\and%
  \'Eric Colin de Verdi\`ere\thanks{Laboratoire d'informatique, \'Ecole
    normale sup\'erieure, CNRS,  Paris, France,
    \protect\url{eric.colin.de.verdiere@ens.fr}}\and%
  Atsuhiro Nakamoto\thanks{Department of Mathematics, Yokohama National University, Yokohama, Japan, \protect\url{nakamoto@ynu.ac.jp}}}

\newcommand{\surf}{\mathcal{S}}

\newcommand{\VM}{V_M}
\newcommand{\BM}{\bar M}
\newcommand{\set}[1]{\{ #1 \}}
\newcommand{\Bigbar}[1]{\mathrel{\left|\vphantom{#1}\right.\n@space}}

\makeatletter
\def\proof{\@ifnextchar[
  {\@xproof}{\@proof}}

\def\@proof{\trivlist \item[\hskip\labelsep {\it Proof.}  ]
  \ignorespaces }
\def\@xproof[#1]{\trivlist \item[\hskip\labelsep{\it #1.}]
  \ignorespaces }
\makeatother

\makeatletter
\def\@toappear{} 

\long\def\toappear#1{\def\@toappear{\parbox[b]{20pc}{\footnotesize #1}}}

\def\ftype@copyrightbox{8}
\def\copyrightspace{
\@float{copyrightbox}[b]
\begin{center}
\setlength{\unitlength}{1pc}
\begin{picture}(20,1.5) 
                                                                                
\put(0,-.95){\@toappear}
\end{picture}
\end{center}
\end@float}
\makeatother

\toappear{This paper appeared in
  \emph{Graphs and Combinatorics}, 2013, 29(6):1675--1688.  The final
  publication is available at \texttt{www.springerlink.com}.}

\begin{document}

\maketitle
\copyrightspace

\begin{abstract}
  A triangulation of a surface is \emph{irreducible} if no edge can be
  contracted to produce a triangulation of the same surface.  In this
  paper, we investigate irreducible triangulations of surfaces with
  boundary. We prove that the number of vertices of an irreducible
  triangulation of a (possibly non-orientable) surface of genus~$g\ge0$
  with $b\ge0$ boundary components is $O(g+b)$. So far, the result was known only
  for surfaces without boundary ($b=0$). While our technique yields a worse
  constant in the $O(.)$ notation, the present proof is elementary, and
  simpler than the previous ones in the case of surfaces without boundary.
\end{abstract}

\section{Introduction}
Let $\surf$ be a surface, possibly with boundary.  A \emph{triangulation}
is a simplicial complex whose underlying space is~$\surf$.  Contracting an
edge of the triangulation (identifying two adjacent vertices in the
simplicial complex) is allowed if this results in another triangulation of
the same surface. A triangulation is \emph{irreducible} (or \emph{minimal})
if no edge is contractible. Every triangulation can be reduced to an
irreducible triangulation by iteratively contracting some of its edges.

Irreducible triangulations have been much studied in the context of
surfaces without boundary.  In this paper, we extend known results to the
case of surfaces with boundary.  Specifically, we prove that the number of
vertices of an irreducible triangulation is linear in the genus and the
number of boundary components of the surface.  Compared to previous works,
our theorem and its proof have two interesting features: the result is more
general, since it applies to surfaces with boundary, and the arguments of
the proof are simpler.

\subsection{Previous Works for Surfaces Without Boundary}

We first describe previous related works, on surfaces without boundary.
Barnette and Edelson~\cite{be-ao2mh-88,be-a2mhf-89} proved that the number
of irreducible triangulations of a given surface is finite. Nakamoto and
Ota~\cite{no-nits-95} were the first to show that the number of vertices in
an irreducible triangulation is at most linear in the
genus of the surface. The best upper bound known to date is due to
Joret and Wood~\cite{jw-its-10}, who proved that this number is at most
$\max\set{13g-4,4}$.  (Here and in the sequel, $g$ is the \emph{Euler
  genus}, which equals twice the usual genus for orientable surfaces and
equals the usual genus for non-orientable surfaces.)  This bound is
asymptotically tight, as there are irreducible triangulations with
$\Omega(g)$ vertices; however, the minimal number of vertices in a
triangulation is $\Theta(\sqrt{g})$~\cite{jr-mtos-80}.

Some low genus cases were studied. Steinitz~\cite{sr-vtp-34} proved that
the unique irreducible triangulation of the sphere is the boundary of the
tetrahedron. The two irreducible triangulations of the projective plane
were found by Barnette~\cite{b-gtpp-82}, followed by the 21~irreducible
triangulations of the torus by Lawrencenko \cite{l-itt-87} and the
29~triangulations of the Klein bottle by Sulanke~\cite{s-nitkb-06}.  More
recently, Sulanke~\cite{s-gits-06,s-itlgs-06} developed a method to
generate all the irreducible triangulations of surfaces without
boundary. His algorithm rediscovered the irreducible triangulations for the
projective plane, the Klein bottle, and the torus; it also built the
irreducible triangulations of the double-torus ({396,784} triangulations)
and the non-orientable surfaces of genus three ({9,708}) and four
({6,297,982}).

Generalizations of the notion of irreducible triangulations, such as
\emph{$k$-irreducible triangulations}, also called \emph{$k$-minimal
  triangulations} ($k\ge3$), have also been
studied~\cite{mn-kts-95,grs-its-96}.  Juvan et
al.~\cite[Section~6]{jmm-scs-96} also study this concept in the case of
surfaces with boundary; their proof technique implies that the number of
irreducible triangulations is finite for every surface (possibly with
boundary).  For a more detailed survey on results on irreducible
triangulations, see Mohar and Thomassen~\cite[Sect.~5.4]{mt-gs-01}.
Higher-dimensional analogs have also been studied, and in particular
conditions ensuring that contracting an edge of a simplicial complex
preserves the topological type~\cite{degn-tpec-99,als-edsrs-11}.

\subsection{Applications of irreducible triangulations}

One motivation for studying irreducible triangulations is that, to solve
some problems on triangulations, it sometimes suffices to solve them on
irreducible triangulations. For example, on a triangulation of an
orientable surface with Euler genus $g\geq 4$ (at least two handles),
Barnette \cite[Conjecture 5.9.3]{mt-gs-01} conjectured that there always
exists a cycle without repeated vertices that is non-null-homotopic and
separating.  More generally, Mohar and Thomassen \cite[Conjecture
5.9.5]{mt-gs-01} conjectured that for every even~$h$, $0<h<g$, there exists
a cycle without repeated vertices that splits the surface into two surfaces
of genus~$h$ and~$g-h$, respectively.  To prove these conjectures, it
suffices to prove them for irreducible triangulations.  (See also the
discussion by Sulanke~\cite[Sect.~5]{s-itlgs-06}.)

Irreducible triangulations have also been used to characterize
triangulations without $K_6$-minors.  (The characterization of abstract
graphs without $K_m$-minors has been done for any $m\leq5$, but the problem
for $m\geq 6$ seems to be very difficult.)  The key fact for the
characterization is that every triangulation on a surface~$\surf$ with no
$K_6$-minor is transformed into an irreducible triangulation with no
$K_6$-minor by contracting edges.  The complete lists of irreducible
triangulations are known only for surfaces of Euler genus at most four, and
so the characterizations are done only for those
surfaces~\cite{kmn-kmtkb-08,mn-kmtcq-09,mnos-kmtns-10,mn-kmtns-12}.

Irreducible triangulations can be used to generate all triangulations of a
given surface~\cite{s-gt2m-91}.  They are also a good tool to study
diagonal flips on triangulations. Negami~\cite{n-dftss-99,n-dfts-94} used
the fact that there are finitely many irreducible triangulations to prove
that two triangulations of a surface with the same number of vertices are
equivalent under diagonal flips, provided the number of vertices is greater
than an integer depending only on the surface.  For further applications,
see the recent paper by Joret and Wood~\cite{jw-its-10} and references
therein.

\subsection{Our Result}

It turns out that the notion of irreducible triangulations extends directly
to the case of surfaces with boundary.  In this paper, we prove that the
number of vertices of such an irreducible triangulation admits an upper
bound that is linear in the genus~$g$ and the number of boundaries~$b$ of
the surface.  In more detail:

\begin{theorem}\label{Th:main}
  Let $\surf$ be a (possibly non-orientable) surface with Euler genus~$g$
  and $b$ boundaries.  Assume $g\geq1$ or $b\geq2$.  Then every irreducible
  triangulation of~$\surf$ has at most $570g+385b-573$ vertices, except for
  the case of the projective plane ($g=1$ and $b=0$), in which the bound
  is~$186$.
\end{theorem}

\begin{figure}
	\begin{center}
		\includegraphics[width=.7\linewidth]{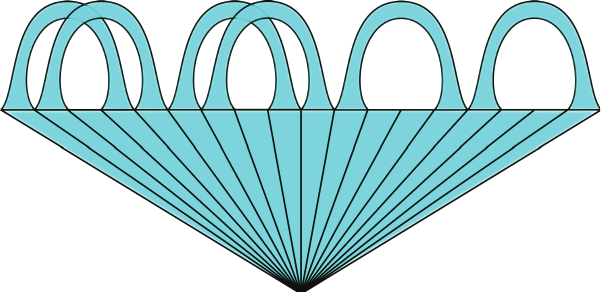}
	\end{center}
	\caption{For any even $g\ge0$ and any $b\ge1$ (one of these two
          inequalities being strict), there exists an irreducible
          triangulation of an orientable surface with Euler genus $g$ with
          $b$ boundary components, and with $5g/2+4b-2$ vertices.  The
          figure illustrates the case $g=4$ and $b=3$.  Starting with a set
          of triangles glued together, all meeting at a vertex (bottom
          part), attach a set of $g/2$ pairs of interlaced rectangular
          strips (top left) and a set of $b-1$ non-interlaced rectangular
          strips (top right), and triangulate every strip by adding an
          arbitrary diagonal (not shown in the picture).  That the
          resulting triangulation is irreducible follows from the fact that
          every edge belongs to a non-null-homotopic 3-cycle or is a
          linking edge (a non-boundary edge whose endpoints are both
          boundary vertices).  Also, note that all vertices are on the
          boundary.  In particular, taking $b=1$, we obtain an irreducible
          triangulation whose single boundary component contains $5g/2+2$
          vertices.}
	\label{F:example}
\end{figure}

This bound is asymptotically tight; see Figure~\ref{F:example}.  Compared
to the case of surfaces without boundary, the main difficulty we
encountered was to prove that the number of boundary vertices is $O(g+b)$
(there are indeed irreducible triangulations of surfaces whose single
boundary contains $\Theta(g)$ vertices, as Figure~\ref{F:example} also
illustrates); the known methods for surfaces without boundary do not seem
to extend easily to surfaces with boundary.  Our strategy is roughly as
follows.  Let $T$ be an irreducible triangulation.  First, we show that
every matching of (the vertex-edge graph of) $T$ has $O(g+b)$ vertices.
Then, we show that every inclusionwise maximal matching contains a constant
fraction of the vertices of~$T$.  For technical reasons, in the case of
surfaces with boundary, we actually need to restrict ourselves to a
matching satisfying some additional mild conditions.

In particular, we reprove that, on a surface without boundary of genus~$g$,
the number of vertices of an irreducible triangulation is~$O(g)$.  Our
method does not improve over the current best bound of $\max\set{13g-4,4}$
by Joret and Wood~\cite{jw-its-10}.  However, it is substantially different
and simpler than the other known proofs of this result.  These former
proofs, by Nakamoto and Ota~\cite{no-nits-95} and Joret and
Wood~\cite{jw-its-10} (see also Gao et al.~\cite{grt-itteo-91}), rely on a
deep theorem by Miller~\cite{m-atgg-87} (see also
Archdeacon~\cite{a-nga-86}) stating that the genus of a graph (the minimum
Euler genus of a surface on which a graph can be embedded) is additive over
2-vertex amalgams (identification of two vertices of disjoint graphs).
 Another paper by Cheng et al.~\cite{cdp-hsmit-04} also claims a linear bound without using Miller's
theorem, but this part of their paper has a flaw (personal communication
with the authors).%
\footnote{Specifically, in the proof of their Lemma~3, the authors
  incorrectly claim that there are at most $g$ pairwise non-homologous
  cycles on an orientable surface of Euler genus~$g$.}  %
In contrast, our proof is short and uses only elementary topological
lemmas.

Finally, we refine the above technique in the case of surfaces without
boundary, and obtain a bound that is better than that of
Theorem~\ref{Th:main}, but no better than the current best result by Joret
and Wood~\cite{jw-its-10}.

\bigskip

We shall introduce some definitions from topology and preliminary lemmas in
Section~\ref{S:prelim}.  We then prove our main theorem
(Section~\ref{S:main}).  Finally, in Section~\ref{S:improv}, we describe
the improvement of the technique for surfaces without boundary.

\section{Preliminaries}\label{S:prelim}

We present a few notions of combinatorial topology; for further details,
see also Stillwell~\cite{s-ctcgt-93}, Armstrong~\cite{a-bt-83}, or
Henle~\cite{h-cit-94}.

\subsection{Topological Background}

\paragraph{Surfaces, Cycles, and Homotopy.}
A \emph{surface} ($2$-manifold with boundary) is a topological Hausdorff
space where each point has an open neighborhood homeomorphic to the plane
or the closed half-plane; the points in the latter case are called
\emph{boundary points}.  Henceforth, $\surf$ denotes a compact, connected
surface.

By the classification theorem, $\surf$ is homeomorphic to a surface
obtained from a sphere by removing finitely many open disks and attaching
handles (\emph{orientable case}) or M\"obius bands (\emph{non-orientable
  case}) along some of the resulting boundaries.  In the orientable case,
the \emph{Euler genus} of~$\surf$, denoted by~$g$, equals \emph{twice} the
number of handles; in the non-orientable case, it equals the number of
M\"obius bands.  The number of remaining \emph{boundary components} is
denoted by~$b$.

In this paper, a \emph{cycle} on~$\surf$ is the image of a one-to-one
continuous map $S^1\to\surf$, where $S^1$ is the standard circle.  In
particular, we emphasize that cycles are undirected and simple.  Two cycles
are \emph{homotopic} if one can be deformed continuously to the other; more
formally, two cycles $C_0$ and~$C_1$ are homotopic if there exists a
continuous map $h:[0,1]\times S^1\to\surf$ such that $h(0,\cdot)$ is
one-to-one and has image~$C_0$, and similarly $h(1,\cdot)$ is one-to-one
and has image~$C_1$.  A cycle is \emph{null-homotopic} if and only if it
bounds a disk on~$\surf$.  We emphasize that only homotopy of cycles is
considered in this paper; for example, we say that two loops are homotopic
if and only if the corresponding cycles are homotopic (without fixing any
point of the loops).

A cycle is \emph{two-sided} if cutting along it results in a (possibly
disconnected) surface with two boundaries, and \emph{one-sided} otherwise.
Equivalently, a cycle is two-sided if it has a neighborhood homeomorphic to
an annulus, and one-sided if it has a neighborhood homeomorphic to a
M\"obius band (which implies that the surface is non-orientable).  Two
cycles in general position \emph{cross} at point~$p$ if they intersect
at~$p$ and the intersection cannot be removed by an arbitrarily small
perturbation of the cycles.  Two homotopic cycles in general position cross
an even number of times if they are two-sided, and an odd number of times
if they are one-sided.

\paragraph{Graph Embeddings, Triangulations, and Edge Contractions.}
Let $G$ be a graph, possibly with loops and multiple edges.  An
\emph{embedding} of~$G$ on~$\surf$ is a ``crossing-free'' drawing of~$G$
on~$\surf$.  More precisely, the vertices of~$G$ are mapped to distinct
points of~$\surf$; each edge is mapped to a path in~$\surf$, meeting the
image of the vertex set only at its endpoints, and such that the endpoints
of the path agree with the points assigned to the vertices of that edge.
Moreover, all the paths must be without intersection or self-intersection
except, of course, at common endpoints.  We sometimes identify~$G$ with its
embedding on~$\surf$.  The \emph{faces} of~$G$ are the connected components
of the complement of the image of~$G$ in~$\surf$.  A graph embedding~$G$
is \emph{cellular} if each of its faces is an open disk.  If it is the
case, \emph{Euler's formula} states that $|V|-|E|+|F|=2-g-b$, where $V$,
$E$, and~$F$ are the sets of vertices, edges, and faces of~$G$,
respectively.

Let $e$ be an edge of a graph $G$ embedded in the interior of~$\surf$.
Assume that $e$ is not a loop.  The \emph{contraction} of $e$ shrinks~$e$
to a single vertex; the resulting graph is in the interior of~$\surf$.
Loops and multiple edges may appear during this process
(Figure~\ref{contractionGraph}).

\begin{figure}
	\begin{center}
		\includegraphics[width=8cm]{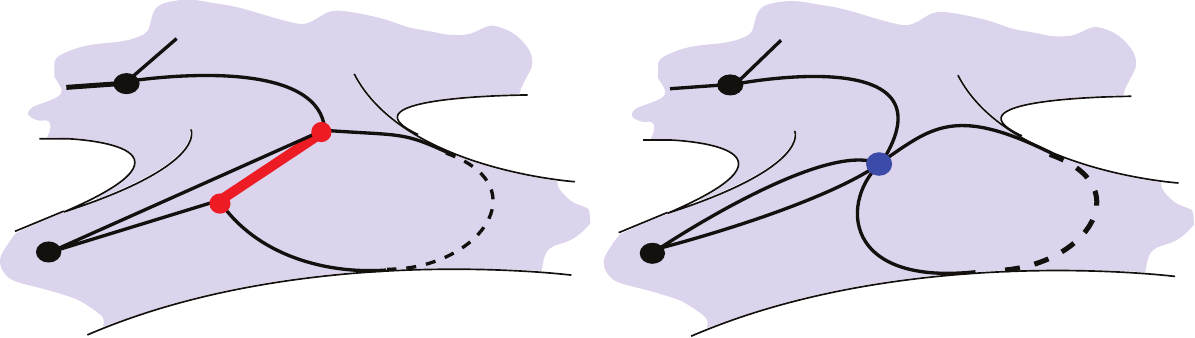}
	\end{center}
	\caption{Edge contraction on an embedded graph.}
	\label{contractionGraph}
\end{figure}

A \emph{triangulation}~$T$ on~$\surf$ is a graph without loops or multiple
edges embedded on~$\surf$ such that each face is an open disk with three
distinct vertices, and such that two such triangles intersect on a single
edge (and its two incident vertices), a single vertex, or not at all.  In
other words, the vertices, edges, and faces of~$G$ form a simplicial
complex whose underlying space is~$\surf$.

The definition of edge contraction on a triangulation is slightly different
from an edge contraction on a graph embedding.  Let $uv$ be an edge of~$T$;
contracting edge~$uv$ identifies both vertices $u$ and~$v$ in the
simplicial complex~$T$; the dimension of some simplices decreases by one.
We say that $uv$ is \emph{contractible} if the new simplicial complex is
still homeomorphic to~$\surf$ (Figure~\ref{contractionTriangulation}).

\begin{figure}
	\begin{center}
		\includegraphics[width=8cm]{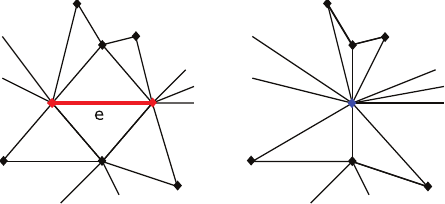}
	\end{center}
	\caption{Edge contraction on a triangulation.}
	\label{contractionTriangulation}
\end{figure}

In more detail, assume for now that $e=uv$ is an interior edge, incident
with triangles $uvx$ and~$uvy$.  Contracting~$e$ shrinks~$e$, identifying
its two vertices $u$ and~$v$, and identifies the two pairs of parallel
edges $\set{ux,vx}$ and~$\set{uy,vy}$.  The definition is similar if $e$ is
a boundary edge, except that it has a single incident triangle~$uvx$.  If
$uv$ is not a boundary edge but exactly one vertex (say~$u$) is incident to
a boundary, then the edge~$uv$ is contracted to~$u$, on the boundary.  If
this operation results in a triangulation of~$\surf$, we say that $e$ is
\emph{contractible}.  In particular, a \emph{linking edge} of~$T$ is a
non-boundary edge whose both vertices are on the boundary; a linking edge
is never contractible.

A triangulation of a surface is \emph{irreducible} if it contains no
contractible edge.  For example, it is known that the only irreducible
triangulation of the sphere is the boundary of a
tetrahedron~\cite{sr-vtp-34}.  Using a similar argument, it is not hard to
show that the only irreducible triangulation of the disk is a
single triangle.

\subsection{Preliminary Lemmas}

We list here a series of basic facts that will be used in our proof.
\begin{lemma}\label{L:noncontr}
  Assume $\surf$ is not the sphere or the disk, and let $T$ be an
  irreducible triangulation of~$\surf$.  Then every non-linking edge of~$T$
  belongs to a non-null-homotopic 3-cycle.
\end{lemma}
\begin{proof}
  This was proved by Barnette and Edelson~\cite[Lemma~1]{be-ao2mh-88} for
  surfaces without boundary: In this case, every edge of~$T$ belongs to a
  3-cycle that is not the boundary of a triangle; if that 3-cycle is
  null-homotopic, then it bounds a disk, and an edge inside that disk must
  be contractible.  The argument immediately extends to non-linking edges
  of surfaces with boundary.  (For boundary edges, we need to distinguish
  whether the boundary has length at least four, in which case the previous
  argument applies, or exactly three, in which case the result is obvious.)\qed
\end{proof}

\begin{lemma}\label{degreeVertex}
  The degree of a non-boundary vertex of an irreducible triangulation
  of~$\surf$ is at least four.
\end{lemma}
\begin{proof}
  This is a direct consequence of a result by Sulanke~\cite[Theorem
  1]{s-gits-06}.  Specifically, he uses Lemma~\ref{L:noncontr} to show that
  every vertex of an irreducible triangulation belongs to two
  non-separating 3-cycles crossing at that vertex.  Again, the argument
  extends to non-boundary vertices of surfaces with boundary.\qed
\end{proof}

\begin{lemma}\label{nonHomotopicCycles}
  Let $G$ be a $1$-vertex graph with $\ell$ loop edges, embedded in the
  interior of~$\surf$.  Assume that no face of~$G$ is a disk bounded by one
  or two edges.  Then $\ell\leq 3g+2b-3$, except if $\surf$ is a sphere or
  a disk (in which cases $\ell=0$).
\end{lemma}
\begin{proof}
  Barnette and Edelson~\cite[Corollary~1]{be-a2mhf-89} prove a similar
  result; see also Chambers et al.~\cite[Lemma 2.1]{ccelw-scsh-08}.  Here
  is a sketch of proof.  Without loss of generality, we can assume that $G$
  is inclusionwise maximal; namely, no edge can be added to~$G$ without
  violating the hypotheses of the lemma.  Then it follows from the
  classification of surfaces that, unless the surface is the sphere, the
  disk, or the projective plane, every face of the graph is a disk bounded
  by three edges, or an annulus bounded by a single edge and a single
  boundary component of~$\surf$.  A standard double-counting argument
  combined with Euler's formula concludes.\qed
\end{proof}
\begin{corollary}\label{nonHomotopicCycles2}
  Let $G$ be a 1-vertex graph with $\ell$ loop edges, embedded in the
  interior of~$\surf$.  Assume that no loop of~$G$ is null-homotopic and
  that no two loops of~$G$ are homotopic. Then $\ell= 0$ if $\surf$ is a 
  sphere or a disk, $ \ell \leq 1$ if $\surf$ is a projective plane, and 
  $\ell \leq 3g+2b-3$ otherwise.
\end{corollary}
\begin{proof}
  The hypotheses imply that no face of~$G$ is a disk bounded by one or two
  edges (and thus Lemma~\ref{nonHomotopicCycles} concludes), unless that
  disk is bounded by twice the same edge (in which case $\surf$ is the
  projective plane).\qed
\end{proof}

\begin{lemma}\label{sizeHomotopyClass} 
  Let $C$ be a non-null-homotopic 3-cycle in an irreducible
  triangulation~$T$ of $\surf$.  No more than nine pairwise edge-disjoint
  3-cycles of~$T$ are homotopic to~$C$.
\end{lemma} 
\begin{proof} 
  \emph{First case: $C$ is two-sided.}  This case is a small variation on a
  lemma by Barnette and Edelson \cite[Lemma 9]{be-ao2mh-88}.  Any two
  distinct 3-cycles homotopic to~$C$ must cross an even number of times,
  hence cannot cross at all; thus two such 3-cycles bound an annulus,
  possibly ``pinched'' on a vertex or an edge.  So the set of 3-cycles
  homotopic to a given 3-cycle can be ordered linearly.  Assume there are
  at least ten pairwise edge-disjoint 3-cycles of~$T$ homotopic to~$C$; let
  us consider ten such consecutive cycles in this ordering,
  $C_1,\ldots,C_{10}$.  See Figure~\ref{homotopyClass}.

  For every~$i$, the annulus between $C_i$ and~$C_{i+3}$ cannot be pinched
  along a vertex: otherwise, it is easy to see that an edge between
  $C_{i+1}$ and~$C_{i+2}$ would be
  contractible~\cite[Lemma~7]{be-ao2mh-88}.  This annulus cannot be pinched
  along an edge, since the cycles are edge-disjoint.  So any two
  consecutive cycles in the sequence $C_1,C_4,C_7,C_{10}$ bound a
  non-pinched annulus.  Now, similarly, some edge between $C_4$ and~$C_7$
  is contractible~\cite[Lemma~9]{be-ao2mh-88}, which is a contradiction.

  \begin{figure}
    \begin{center}
      \includegraphics[width=8cm]{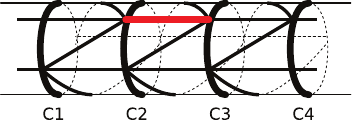} 
    \end{center}
    \caption{Illustration of the proof of Lemma~\ref{sizeHomotopyClass} (in
      the two-sided case): if there are ten homotopic edge-disjoint 3-cycles,
      there must be four pairwise disjoint homotopic 3-cycles, so there is at
      least one contractible edge.}
    \label{homotopyClass}
  \end{figure}
  
  \medskip

  \emph{Second case: $C$ is one-sided.}  In this case, any 3-cycle
  homotopic to~$C$ crosses~$C$ exactly once, and must therefore share
  exactly one vertex with~$C$.  Let $v$ be any vertex of~$C$; we prove
  below that at most two 3-cycles different from~$C$ and homotopic to~$C$
  pass through~$v$.  This proves that there are at most seven 3-cycles
  homotopic to~$C$ (including~$C$ itself), which concludes.

  So assume that (at least) four 3-cycles (including~$C$) homotopic to~$C$
  share together a vertex~$v$.  These cycles lie in a M\"obius band
  ``pinched'' at~$v$, and we can order them linearly; let $C_1,\ldots,C_4$
  be consecutive cycles in this ordering.  As in the first
  case~\cite[Lemma~7]{be-ao2mh-88}, an edge between $C_2$ and~$C_3$ would
  be contractible, a contradiction.\qed
\end{proof}

\section{Proof of Theorem~\ref{Th:main}}\label{S:main}

A \emph{matching}~$M$ of a graph~$G$ is a set of edges of~$G$ such that
every vertex of~$G$ belongs to at most one edge of~$M$.

\subsection{The Size of a Matching}

Our first task is to prove that a matching of an irreducible triangulation
has size $O(g+b)$.

\begin{proposition}\label{sizeMatching}
  Let $T$ be an irreducible triangulation of~$\surf$, where $g\geq 1$ or
  $b\geq 2$.  Let $M$ be a matching of~$T$ containing no linking edge.
  Then the number of edges of~$M$ is at most~$27$ if $\surf$ is the
  projective plane ($g=1$ and $b=0$) and $81g+54b-81$ otherwise.
\end{proposition}
\begin{proof}
  The structure of the proof is as follows.  In three steps, we remove
  edges from~$M$, obtaining successive matchings~$M_1$, $M_2$, and~$M_3$,
  each of them satisfying additional properties.  We show that the edge set
  of~$M_3$ is in bijection with the edge set of a 1-vertex graph
  on~$\surf$ where no edge is null-homotopic and no two edges are
  homotopic.  By Corollary~\ref{nonHomotopicCycles2}, this implies that
  $M_3$ has $O(g+b)$ edges.  Furthermore, we show that $M_3$ contains some
  constant fraction of the edges of~$M$, so that $M$ also has $O(g+b)$
  edges.

  Recall that every edge~$e$ of~$M$ belongs to a non-null-homotopic 3-cycle
  (Lemma~\ref{L:noncontr}); let~$C_e$ be such a cycle.

  \medskip

  \emph{Construction of~$M_1$.}  Assume that there are three distinct edges
  $e_1$, $e_2$, and~$e_3$ of~$M$ such that $C_{e_1}$ shares an edge with
  both $C_{e_2}$ and~$C_{e_3}$.  Then $C_{e_1}\cup C_{e_2}\cup C_{e_3}$ has
  at most 5 vertices, implying that two of $e_1,e_2,e_3$ have an endpoint
  in common.  This contradiction proves that every 3-cycle~$C_{e_1}$ shares
  an edge of~$T$ with at most one other 3-cycle~$C_{e_2}$.  Let $M_1$ be
  obtained from~$M$ by removing $e_1$ or~$e_2$ for every such pair of edges
  $\{e_1,e_2\}$.  By the previous property, the cycles $C_e$, $e\in M_1$,
  are edge-disjoint.  The set $M_1$ satisfies the hypotheses of the lemma,
  and $|M|\leq2|M_1|$.  Now, we forget~$M$ and focus on bounding the size
  of~$M_1$.

  \medskip

  \emph{Construction of~$M_2$.} We partition the edges~$e$ of~$M_1$
  according to the homotopy class of the corresponding 3-cycle~$C_e$.
  Let $M_2$ be obtained by choosing one arbitrary representative edge per
  class; the cycles $C_e$, $e\in M_2$, are in distinct homotopy
  classes.  We have $|M_1|\leq 9|M_2|$ by Lemma~\ref{sizeHomotopyClass} and
  since the cycles $C_e$, $e\in M_1$, are edge-disjoint.  Now, the
  cycles $C_e$, $e\in M_2$, are in distinct non-trivial homotopy
  classes and are edge-disjoint.

  \medskip

  \emph{Construction of~$M_3$.}  For every $e\in M_2$, let $\pi_e$ be the
  path of length two obtained from $C_e$ by removing $e$. We orient
  the two edges of~$\pi_e$ towards the extremities of~$\pi_e$ (which are
  also the endpoints of $e\in M_2$).  Since $M_2$ is a matching, every
  vertex of the triangulation~$T$ is the target of at most one oriented
  edge.

  \begin{figure}
    \begin{center}
      \includegraphics[width=8cm]{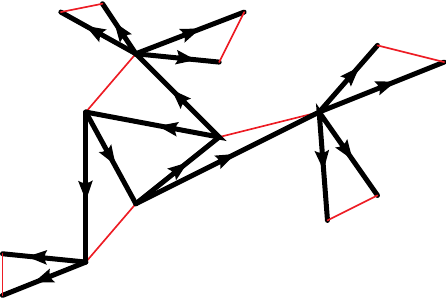}
    \end{center}
    \caption{In light lines, the matching~$M_2$; in bold lines, the
      graph~$\Pi_2$, here forming a tree plus an edge.}
    \label{graph}
  \end{figure}

  Let $\Pi_2$ be the union of the graphs $\pi_e$, $e\in M_2$.  We claim
  that $\Pi_2$ is a \emph{pseudoforest}: every connected component~$\Pi'_2$
  of~$\Pi_2$ contains at most one cycle (see Figure~\ref{graph}).  Indeed,
  every vertex of~$\Pi'_2$ is the target of at most one oriented edge, so
  the number of edges of~$\Pi'_2$ is at most the number of vertices
  of~$\Pi'_2$, and $\Pi'_2$ is connected; so a spanning tree of~$\Pi'_2$
  contains all but at most one edge of~$\Pi'_2$.

  If $\Pi'_2$ is not a tree, let $e'$ be an edge such that $\Pi'_2-e'$ is a
  tree.  The edge~$e'$ belongs to some~$\pi_e$.  We remove $e$ from~$M_2$
  (and consequently $\pi_e$ from~$\Pi'_2$); the graph $\Pi'_2$ becomes one
  or two trees, and the other connected components of~$\Pi_2$ are
  unaffected.  We do this iteratively for every connected
  component~$\Pi'_2$ of~$\Pi_2$.  Let $M_3$ be obtained from~$M_2$ after
  removing these edges.

  Before the removal of any edge of~$M_2$, if a connected component
  $\Pi'_2$ of~$\Pi_2$ is not a tree, it contains a cycle~$C$ of length at
  least three (Figure \ref{graph}); since each vertex of~$C$ has at most
  one incoming edge, the edges of~$C$ belong to distinct~$C_e$, for $e\in
  M_2$.  Therefore, when removing edges of~$M_2$ to form~$M_3$, we remove
  at most one third of the edges of~$M_2$.  So $|M_2|\leq\frac32|M_3|$.

  \medskip

  \emph{End of the proof.}  Let $\Pi_3$ be the union of the graphs $\pi_e$,
  for $e\in M_3$; by construction, $\Pi_3$ is a forest.  We now view~$T$ as
  a graph embedded on~$\surf$ (slightly moving it towards the interior
  of~$\surf$, if $\surf$ has a boundary), and contract all the edges
  of~$\Pi_3$ in this graph; this is legal since this set contains no cycle.
  Each edge~$e$ of~$M_3$ is transformed into a loop~$\ell_e$ homotopic
  to~$C_e$.  The loops $\ell_e$ form a graph~$\Gamma$ embedded on~$\surf$;
  that graph has a single vertex per connected component.  There exists a
  tree~$U$ embedded on~$\surf$ meeting~$\Gamma$ exactly at its vertex set.
  We may contract~$U$ on the surface; now $\Gamma$ is transformed into a
  set of simple, pairwise disjoint loops~$\Gamma'$ with the same vertex.
  Furthermore, the loops are non-null-homotopic and pairwise non-homotopic,
  so Corollary~\ref{nonHomotopicCycles2} implies that
  $|M_3|=|\Gamma'|\leq3g+2b-3$ (unless $\surf$ is the projective plane, in
  which case the upper bound is one).  By construction, we also have
  $|M|\leq 2|M_1|\leq 2*9|M_2|\leq 2*9*\frac32|M_3|=27|M_3|$, which
  concludes the proof.\qed
\end{proof}

\subsection{An Inclusionwise Maximal Matching Covers Many Vertices}

Now, we prove that an inclusionwise maximal matching of~$T$ must cover a
constant fraction of the edges of~$T$.

\begin{proposition}\label{verticesBound}
  Let $T$ be an irreducible triangulation of~$\surf$.  Let $W$ be a set of
  vertices of~$T$.  Let $M$ be an inclusionwise maximal matching of~$T$
  among those that avoid~$W$.  Assume further that every boundary vertex
  of~$T$ is either in~$W$ or incident to an edge of~$M$.  Then the number
  of vertices of~$T$ is at most $7|M|+4|W|+3g+3b-6$.
\end{proposition}
\begin{proof}
  Let us denote by $V$, $E$, and~$F$ the vertices, edges, and faces of~$T$,
  respectively.  Let $\VM$ be the vertices reached by~$M$ and $X$ be the
  vertices neither in~$\VM$ nor in~$W$. Let $\bar{M}$ be the set of the
  edges of~$T$ that are not in~$M$.  Thus $\set{W,\VM,X}$ is a partition
  of~$V$, and $\set{M,\BM}$ is a partition of~$E$.

  Let $v\in X$.  Recall that $v$ is a non-boundary vertex by hypothesis.
  According to Lemma~\ref{degreeVertex}, $v$ has degree at least four, so
  it is incident to at least four edges in~$\BM$.  By the maximality of the
  matching~$M$, the other vertex of each of these four edges is not
  in~$X$.  So, charging each vertex $v$ of~$X$ with these four edges,
  we obtain that $4|X|\leq |\BM|$.  

  The rest of the proof is standard machinery.  Since $T$ is a
  triangulation, by double-counting we obtain $|F|\leq\frac23|E|$ (this is
  not an equality in general since $\surf$ may have boundary).  Plugging
  this relation into Euler's formula $|V|-|E|+|F|=2-g-b$, we obtain:
  \[(|W|+|\VM|+|X|)-\frac13(|M|+|\BM|)\geq 2-g-b.\] %
  $|V_M|=2|M|$ gives, after some rearranging:
  \[|\BM|-3|X|\leq 5|M|+3|W|+3g+3b-6.\]%
  As shown above, $4|X|\leq |\BM|$, implying $|X|\leq|\BM|-3|X|$ and so
  \[|X|\leq5|M|+3|W|+3g+3b-6.\]%
  This bound on~$|X|$ allows to bound $|V|=2|M|+|W|+|X|$ in terms of $|M|$,
  $|W|$, $g$, and~$b$, implying the result.\qed
\end{proof}

\subsection{End of Proof}

The proof of Theorem~\ref{Th:main} combines Propositions~\ref{sizeMatching}
and~\ref{verticesBound}:

\begin{proof}[Proof of Theorem~\ref{Th:main}]
  Let $W$ be a set of vertices, one on each boundary component of~$\surf$
  having an odd number of vertices.  Build a matching~$M$ made of edges on
  the boundary of~$\surf$ and covering the vertices on the boundary
  of~$\surf$ that are not in~$W$.  Extend~$M$ to an inclusionwise maximal
  matching of~$T$ that avoids~$W$; we still denote it by~$M$.

  $M$ contains no linking edge by construction so, by
  Proposition~\ref{sizeMatching}, $M$ has less than $81g+54b-81$ edges
  ($27$ if $\surf$ is the projective plane).  By
  Proposition~\ref{verticesBound}, and since $|W|\leq b$, the number of
  vertices of~$T$ is at most $7|M|+3g+7b-6$.

  Combining these equations proves that $T$ has at most $570g+385b-573$
  vertices ($186$ if $\surf$ is the projective plane).\qed
\end{proof}

\section{Improvement for Surfaces Without Boundary}\label{S:improv}

The purpose of this section is to improve the previous bound when $\surf$
has no boundary ($b=0$).  The strategy is to improve the bound of
Proposition~\ref{verticesBound} using a more careful analysis.

\begin{theorem}\label{Th:main2}
  Let $\surf$ be a (possibly non-orientable) surface with Euler
  genus~$g\ge1$ and without boundary.  Then every irreducible triangulation
  of~$\surf$ has at most $f(g)$ vertices, where $f(1)=55$, $f(2)=194$,
  $f(3)=333$, and $f(g)=163g-164$ if $g\ge4$.
\end{theorem}

The following lemma appears in an article by Fujisawa et
al.~\cite[Sect.~2]{fno-hcbtg-13}; we reproduce the proof in more
detail here for convenience.
\begin{lemma}\label{lem:conncomp}
  Let $\surf$ be a surface of Euler genus~$g\ge1$ without boundary, and let
  $G=(V,E)$ be a 4-connected graph embedded on~$\surf$.  Then for every
  $U\subseteq V$, the number of components of $G-U$ is at most
  $\max\set{1,|U|+g-2}$.
\end{lemma}
\begin{proof}
  We can assume that $U\neq\emptyset$ and that $G-U$ has at least two
  connected components; otherwise, the result is clear.  Let $K$ be the
  graph obtained from~$G$ by the following steps:
  \begin{enumerate}
  \item Contract the edges of a spanning forest of $G-U$.  Now the current
    graph has vertex set~$U\cup W$, where $W$ has one element for each
    component of $G-U$.  The following steps will only add and remove edges
    of this graph.
  \item Delete each edge with both endpoints in~$U$.  Similarly, delete
    each edge with both endpoints in~$G-U$ (such edges are actually loops,
    by the first step).  Now the current graph is bipartite.
  \item On each face of the resulting graph that is not a disk, add edges
    to cut that face into a disk.  This can be done without violating
    bipartiteness, because every face has a boundary component with at
    least one vertex in~$U$ and at least one vertex in~$W$ (since $U$
    and~$W$ are non-empty).
  \item If there exists a face incident with exactly two edges, remove one
    of these two edges.  (The two edges incident to the face are distinct,
    because $\surf$ is not the sphere and the edge is not a loop.)  Repeat
    this step as much as possible.
  \end{enumerate}
  We now have:
  \begin{eqnarray*}
    4|W| & \le & |E(K)|\\
    & \le & 2(|E(K)|-|F(K)|)\\
    & = & 2(|W|+|U|+g-2).\\
  \end{eqnarray*}
  Indeed, the first inequality holds by 4-connectivity of~$G$: since
  $|W|\ge2$, every component of $G-U$ is adjacent to at least four
  different vertices of~$U$; therefore, in~$K$, every vertex of~$W$ is
  adjacent to at least four different vertices of~$U$.  The second line
  follows from the fact that each face is incident to at least four edges
  (by bipartiteness of~$K$ and using Step~4).  The third line holds by
  virtue of Euler's formula, since $K$ is cellularly embedded on~$\surf$.\qed
\end{proof}

\begin{proposition}\label{verticesBound2}
  Let $\surf$ be a surface of Euler genus~$g\geq1$ without boundary, and
  let $G=(V,E)$ be a 4-connected graph embedded on~$\surf$.  Let $M$ be a
  maximum-size matching of~$G$.  Then the number of vertices of~$G$ is at
  most $2|M|+\max\set{1,g-2}$.
\end{proposition}
\begin{proof}
  The Tutte-Berge
  formula~\cite{b-cmg-58,w-spbtf-11}\cite[Sect.~24.1]{s-cope-03} asserts
  that the number of vertices of~$G$ not covered by a maximum-size matching
  of~$G$ is the maximum, over all $U\subseteq V$, of $o(G-U)-|U|$, where
  $o(G-U)$ denotes the the number of components of the graph $G-U$ with an
  odd number of vertices.  By Lemma~\ref{lem:conncomp}, for every
  $U\subseteq V$, we have $o(G-U)-|U|\le\max\set{1,g-2}$.  The result
  follows.\qed
\end{proof}

\begin{proof}[Proof of Theorem~\ref{Th:main2}]
  If $T$ is 4-connected, by Proposition~\ref{verticesBound2}, $T$ has at
  most $2|M|+\max\set{1,g-2}$ vertices where $M$ is a maximum-size matching
  of~$T$.  Using the bound on the size of a maximal matching~$M$
  (Proposition~\ref{sizeMatching}), we deduce that $T$ has at most $h(g)$
  vertices, where $h(1)=55$, $h(2)=163$, and $h(g)=163g-164$ if $g\ge3$.

  If $T$ is not 4-connected, this means that a vertex set~$U$ of size at
  most three separates~$T$.  Actually, $|U|=3$, and $U$ forms a 3-cycle~$C$
  in~$T$.  This cycle~$C$ must be separating, but also non-null-homotopic,
  for otherwise some edge of~$T$ would be contractible (as in the proof of
  Lemma~\ref{L:noncontr}).  Let $\surf_1$ and~$\surf_2$ be the surfaces
  obtained by cutting~$\surf$ along~$C$ and attaching a triangle to each
  copy of~$C$.  The Euler genera of $\surf_1$ and~$\surf_2$ add up to~$g$.
  Furthermore, $C$ is two-sided (since it is separating), so the number of
  3-cycles homotopic to~$C$ in~$\surf$ is at most~$27$
  \cite[Lemma~9]{be-ao2mh-88}.  Any edge that is contractible in $\surf_1$
  or~$\surf_2$ belongs to such a cycle.  So the total number of edges
  in~$\surf_1$ and~$\surf_2$ that are contractible is at most
  $3\times27+3=84$ (the ``$+3$'' term comes from the fact that the three
  edges of~$C$ may be contractible in both $\surf_1$ and~$\surf_2$.)  A
  similar reasoning is used by Barnette and Edelson \cite[Proof of
  Theorem~2]{be-ao2mh-88}.

  It follows that the number of vertices of an irreducible triangulation of
  a surface without boundary with Euler genus~$g$ is bounded from above by
  $f(g)$, where $f$ satisfies the induction formula:
  \[f(g)=\max\left\{h(g),\max_{\substack{g_1+g_2=g\\g_1,g_2\ge1}}\big\{f(g_1)+f(g_2)+84\big\}\right\}.\]
  Thus, we have $f(1)=55$, $f(2)=194$, and for $g\ge3$:
  \[f(g)=\max\left\{163g-164,\max_{\substack{g_1+g_2=g\\g_1,g_2\ge1}}\big\{f(g_1)+f(g_2)+84\big\}\right\}.\]
  It is easily checked by induction that $f(3)=333$ and $f(g)=163g-164$ for
  $g\ge4$.\qed
\end{proof}

\paragraph*{Acknowledgment.}  We would like to thank Gwena\"el Joret, Thom
Sulanke, and David Wood for helpful comments which improved the readability
of this article.

\makeatletter\@ifundefined{url}{\def\url#1{\texttt{#1}}}{}\@ifundefined{prebib%
}{\relax}{\prebib}\makeatother

\end{document}